\documentclass[12pt,leqno]{article}
\usepackage{amsfonts}
\usepackage{amsmath, dsfont, amsthm, amsfonts, amssymb, color}
\usepackage{mathrsfs}
\usepackage{color}
\usepackage{stmaryrd}
\usepackage[notcite,notref]{showkeys}

\newtheorem{thm}{Theorem}[section]
\newtheorem{cor}[thm]{Corollary}
\newtheorem{lem}[thm]{Lemma}

\theoremstyle{definition}

\newcommand{\scr}[1]{\mathscr #1}
\definecolor{wco}{rgb}{0.5,0.2,0.3}

\numberwithin{equation}{section} \theoremstyle{remark}

\newcommand{\ua}{\uparrow}

\title{{\bf
Uniform in Time Propagation of Chaos for Mean Field Particle System with Interacting Noise and Partially Dissipative Drifts }\footnote{Supported in
 part by  National Key R\&D Program of China (No. 2022YFA1006000) and NNSFC (12271398).} }
\author{
{\bf   Xing Huang  }\\
\footnotesize{ Center for Applied Mathematics, Tianjin
University, Tianjin 300072, China}\\
\footnotesize{  xinghuang@tju.edu.cn}}
\begin{document}
\allowdisplaybreaks
\def\R{\mathbb R}  \def\ff{\frac} \def\ss{\sqrt} \def\B{\mathbf
B} \def\W{\mathbb W}
\def\N{\mathbb N} \def\kk{\kappa} \def\m{{\bf m}}
\def\ee{\varepsilon}\def\ddd{D^*}
\def\dd{\delta} \def\DD{\Delta} \def\vv{\varepsilon} \def\rr{\rho}
\def\<{\langle} \def\>{\rangle} \def\GG{\Gamma} \def\gg{\gamma}
  \def\nn{\nabla} \def\pp{\partial} \def\E{\mathbb E}
\def\d{\text{\rm{d}}} \def\bb{\beta} \def\aa{\alpha} \def\D{\scr D}
  \def\si{\sigma} \def\ess{\text{\rm{ess}}}
\def\beg{\begin} \def\beq{\begin{equation}}  \def\F{\scr F}
\def\Ric{\text{\rm{Ric}}} \def\Hess{\text{\rm{Hess}}}
\def\e{\text{\rm{e}}} \def\ua{\underline a} \def\OO{\Omega}  \def\oo{\omega}
 \def\tt{\tilde} \def\Ric{\text{\rm{Ric}}}
\def\cut{\text{\rm{cut}}} \def\P{\mathbb P} \def\ifn{I_n(f^{\bigotimes n})}
\def\C{\scr C}      \def\aaa{\mathbf{r}}     \def\r{r}
\def\gap{\text{\rm{gap}}} \def\prr{\pi_{{\bf m},\varrho}}  \def\r{\mathbf r}
\def\Z{\mathbb Z} \def\vrr{\varrho}
\def\L{\scr L}\def\Tt{\tt} \def\TT{\tt}\def\II{\mathbb I}
\def\i{{\rm in}}\def\Sect{{\rm Sect}}  \def\H{\mathbb H}
\def\M{\scr M}\def\Q{\mathbb Q} \def\texto{\text{o}}
\def\Rank{{\rm Rank}} \def\B{\scr B} \def\i{{\rm i}} \def\HR{\hat{\R}^d}
\def\to{\rightarrow}\def\l{\ell}\def\iint{\int}
\def\EE{\scr E}\def\Cut{{\rm Cut}}
\def\A{\scr A} \def\Lip{{\rm Lip}}
\def\BB{\scr B}\def\Ent{{\rm Ent}}\def\L{\scr L}
\def\R{\mathbb R}  \def\ff{\frac} \def\ss{\sqrt} \def\B{\mathbf
B}
\def\N{\mathbb N} \def\kk{\kappa} \def\m{{\bf m}}
\def\dd{\delta} \def\DD{\Delta} \def\vv{\varepsilon} \def\rr{\rho}
\def\<{\langle} \def\>{\rangle} \def\GG{\Gamma} \def\gg{\gamma}
  \def\nn{\nabla} \def\pp{\partial} \def\E{\mathbb E}
\def\d{\text{\rm{d}}} \def\bb{\beta} \def\aa{\alpha} \def\D{\scr D}
  \def\si{\sigma} \def\ess{\text{\rm{ess}}}
\def\beg{\begin} \def\beq{\begin{equation}}  \def\F{\scr F}
\def\Ric{\text{\rm{Ric}}} \def\Hess{\text{\rm{Hess}}}
\def\e{\text{\rm{e}}} \def\ua{\underline a} \def\OO{\Omega}  \def\oo{\omega}
 \def\tt{\tilde} \def\Ric{\text{\rm{Ric}}}
\def\cut{\text{\rm{cut}}} \def\P{\mathbb P} \def\ifn{I_n(f^{\bigotimes n})}
\def\C{\scr C}      \def\aaa{\mathbf{r}}     \def\r{r}
\def\gap{\text{\rm{gap}}} \def\prr{\pi_{{\bf m},\varrho}}  \def\r{\mathbf r}
\def\Z{\mathbb Z} \def\vrr{\varrho}
\def\L{\scr L}\def\Tt{\tt} \def\TT{\tt}\def\II{\mathbb I}
\def\i{{\rm in}}\def\Sect{{\rm Sect}}  \def\H{\mathbb H}
\def\M{\scr M}\def\Q{\mathbb Q} \def\texto{\text{o}} \def\LL{\Lambda}
\def\Rank{{\rm Rank}} \def\B{\scr B} \def\i{{\rm i}} \def\HR{\hat{\R}^d}
\def\to{\rightarrow}\def\l{\ell}
\def\8{\infty}\def\I{1}\def\U{\scr U} \def\n{{\mathbf n}}
\maketitle

\begin{abstract} In this paper, uniform in time quantitative propagation of chaos in $L^1$-Wasserstein distance for mean field interacting particle system is derived, where the diffusion coefficient is allowed to be interacting and the drift is assumed to be partially dissipative. The main tool relies on reflection coupling, the gradient estimate of the decoupled SDEs, and the Duhamel formula for two semigroups associated to two time-inhomogeneous diffusion processes on $(\R^d)^N$.
 \end{abstract}

 \vskip 2cm
\section{Introduction}
Kac's chaotic property, also called the Boltzmann property, is important to derive the space homogeneous Boltzmann equation in \cite{Kac}. From the propagation of chaos for mean field interacting particle system, i.e. the dynamic evolution of Kac's chaotic property with respect to the time variable, see for instance \cite{SZ},  one can see that the limit equation of a single particle is the McKean-Vlasov SDE, which is proposed in \cite{McKean}. The McKean-Vlasov SDEs, also named have been systematically  investigated in the recent monograph \cite{WR2024}, where the well-posedness, log-Harnack inequality(equivalently, the entropy-cost inequality), Bismut type derivative formula as well as the ergodicity are established and the models also involve in killing and reflecting cases.

In the mean filed interacting particle system, we usually adopt the following distance on $(\R^d)^m$ for $m\geq 1$:
\begin{align}\label{hdi}\|x-y\|_{1,1}=\sum_{i=1}^m|x^i-y^i|,\ \ x=(x^1,x^2,\cdots,x^m), y=(y^1,y^2,\cdots,y^m)\in(\R^d)^m,
\end{align}
where $|\cdot|$ is the Euclidean distance on $\R^d$.
Let $\scr P((\R^d)^m)$ be the set of all probability measures on $(\R^d)^m$ equipped with the weak topology. Let
$$\scr P_1((\R^d)^m):=\big\{\mu\in \scr P((\R^d)^m): \mu(\|\cdot\|_{1,1})<\infty\big\},$$
which is a Polish space under the $L^1$-Wasserstein distance
$$\W_1(\gamma,\tilde{\gamma})=\inf_{\pi\in\mathbf{C}(\gamma,\tilde{\gamma})}\int_{(\R^d)^m\times (\R^d)^m}\|x-y\|_{1,1}\pi(\d x,\d y),\ \ \gamma,\tilde{\gamma}\in\scr P_1((\R^d)^m)$$ where $\mathbf{C}(\gamma,\tilde{\gamma})$ is the set of all couplings of $\gamma$ and $\tilde{\gamma}$.
Moreover, the Kantorovich dual formula
\begin{align}\label{wetad}\W_1(\gamma,\tilde{\gamma})=\sup_{[f]_{1,1}\leq 1}|\gamma(f)-\tilde{\gamma}(f)|,\ \ \gamma,\tilde{\gamma}\in \scr P_1((\R^d)^m)
\end{align}
holds for $[f]_{1,1}:=\sup_{x\neq y}\frac{|f(x)-f(y)|}{\|x-y\|_{1,1}}$.

%We will also use the total variation distance:
%$$\|\gamma-\tilde{\gamma}\|_{var}=\sup_{\|f\|_\infty\leq 1}|\gamma(f)-\tilde{\gamma}(f)|,\ \ \gamma,\tilde{\gamma}\in \scr P(E).$$

Let $\{(B_t^i)_{t\geq 0}\}_{i\geq 1}$ be independent $n$-dimensional Brownian motions on some complete filtration probability space $(\Omega, \scr F, (\scr F_t)_{t\geq 0},\P)$ and $(X_0^i)_{i\geq 1}$ be i.i.d. $\F_0$-measurable $\R^d$-valued random variables.
%Consider McKean-Vlasov SDEs:
%\begin{align*}\d  X_t=b_t(X_t,\L_{X_t})\mathrm{d} t+\sigma_t(X_t,\L_{X_t})\mathrm{d} W_t,
%\end{align*}
%where $\L_{X_t}$ is the distribution of $X_t$,
Let $b:\R^d\times\scr P(\R^d)\to\R^d$, $\sigma:\R^d\times\scr P(\R^d)\to\R^d\otimes\R^{n}$ be measurable and bounded on bounded sets.
Let $N\ge1$ be an integer. Consider the mean field interacting particle system:
\begin{align*}
\d X^{i,N}_t=b(X_t^{i,N}, \hat\mu_t^N)\d t+\sigma(X^{i,N}_t, \hat\mu_t^N) \d B^i_t,\ \ 1\leq i\leq N,
\end{align*}
where $\hat\mu_t^N =\ff{1}{N}\sum_{j=1}^N\dd_{X_t^{j,N}}$ is the empirical distribution of $(X_t^{i,N})_{1\leq i\leq N}$ and the distribution of $(X_0^{i,N})_{1\leq i\leq N}$ is exchangeable, i.e. for any permutation $\{i_k:1\leq k\leq N\}$ of $\{k:1\leq k\leq N\}$, $(X_0^{i_k,N})_{1\leq k\leq N}$ is identically distributed with $(X_0^{i,N})_{1\leq i\leq N}$.
We also consider the independent McKean-Vlasov SDEs:
\begin{align*}\d X_t^i= b(X_t^i, \L_{X_t^i})\d t+  \sigma(X^i_t,\L_{X_t^i}) \d B^i_t,\ \ 1\leq i\leq N
\end{align*}
for $\L_{X_t^i}$ being the distribution of $X_t^i$.

There are plentiful results on uniform in time propagation of chaos for mean field interacting particle system. When $b(x,\mu)=\nabla U(x)+\int_{\R^d}\nabla W(x-y)\mu(\d y)$, $\sigma=I_{d\times d}$, the author in \cite{M} uses the synchronous coupling method to derive the uniform in time propagation of chaos in strong sense, i.e.
$$\sup_{t\geq 0}\E|X_t^{1,N}-X_t^{1}|^2\leq \frac{c}{N}$$
holds for some constant $c>0$,
where $U$ and $W$ are uniformly convex and $X_0^{i,N}=X_0^i, 1\leq i\leq N$. Meanwhile, the  uniform in time propagation of chaos in relative entropy is also obtained by the Bakry-Emery curvature condition for the time-inhomogeneous decoupled SDEs.

When $\sigma=I_{d\times d}$, $\R^d=T^d$, the authors in \cite{GBM} combine the entropy method introduced in \cite{BJW,JW,JW1} with the uniform in time log-Sobolev inequality for $\L_{X_t^i}$ to derive the quantitative entropy-entropy type propagation of chaos for mean field particle system with singular interaction kernel.
By the technique of BBGKY hierarchy in \cite{L21} as well as the uniform in time log-Sobolev inequality for $\L_{X_t^i}$, the authors in \cite{LL} establish the sharp rate of entropy-entropy type propagation of chaos for particle system with bounded or Lipschitz continuous interaction kernel, see also \cite{MRW} for the explicit conditions for uniform in time log-Sobolev inequality for $\L_{X_t^i}$. Furthermore, \cite{CLY} considers the conditional propagation of chaos in $\W_p(p\geq 2)$-distance for mean field interacting particle system with common noise.

Still in the case $\sigma=I_{d\times d}$, to establish the uniform in time propagation of chaos in $L^1$-Wasserstein distance, the asymptotic reflection coupling is applied in \cite{DEGZ,GBMEJP,LWZ}. \cite{HYY} extends the results in \cite{DEGZ} to the multiplicative noise case. %\cite{LGYG} gets the quantitative long time propagation of chaos in strong sense under the uniformly dissipative condition.
In the additive L\'{e}vy noise case, \cite{LMW} adopts the asymptotic refined basic coupling to derive the uniform in time propagation of chaos in $L^1$-Wasserstein distance. The drifts in \cite{DEGZ,GBMEJP,HYY,LMW,LWZ} are only assumed to be  partially dissipative. One can refer to \cite{W15} for more details on asymptotic reflection coupling.
We also mention that the present author proves the long time entropy-cost type propagation of chaos in the multiplicative noise frame in \cite{HX23e}, where the propagation of chaos in relative entropy depend on the Wasserstein distance between $(X_0^{i,N})_{1\leq i\leq N}$ and $(X_0^{i})_{1\leq i\leq N}$, which allows $\L_{(X_0^{i,N})_{1\leq i\leq N}}$ to be singular with $\L_{(X_0^{i})_{1\leq i\leq N}}$.

However, to the best of our knowledge, there is no result on the uniform in time propagation of chaos in $\W_1$ for mean field particle system with interacting diffusion coefficients and partially dissipative drifts. In this paper, we will try to make some contributions in this topic. A well-known model with interacting diffusion coefficients is the Landau equation.

Throughout the paper, we will consider the following mean field particle system with interacting noise:
\begin{align}\label{Inm}
\nonumber\d X_t^{i,N}&=b^{(0)}(X_t^{i,N})\d t+\frac{1}{N}\sum_{j=1}^Nb^{(1)}(X_t^{i,N},X_t^{j,N})\d t\\
&+\sqrt{\beta}\d
W_t^i
+\frac{1}{N}\sum_{j=1}^N\tilde{\sigma}(\tilde{X}_t^{i,N}, \tilde{X}_t^{j,N})\d B_t^i,\ \ 1\leq i\leq N,
\end{align}
where $\{(W_t^i)_{t\geq 0}\}_{i\geq 1}$ are independent $d$-dimensional Brownian motions, which are independent of  $\{(B_t^i)_{t\geq 0}\}_{i\geq 1}$, $b^{(0)}:\R^d\to\R^d$, $b^{(1)}:\R^d\times\R^d\to\R^d$, $\tilde{\sigma}:\R^d\times\R^d\to\R^d\otimes\R^{n}$ are measurable and bounded on bounded sets, $\beta>0$ is a constant and $b^{(0)}$ satisfies partially dissipative condition \eqref{pdi} below.
Correspondingly, the independent McKean-Vlasov SDEs are formulated as
\begin{equation}\label{MVi}\begin{split}
\d X_t^{i}&=b^{(0)}(X_t^i)\d t+\int_{\R^d}b^{(1)}(X_t^i,y)\L_{X_t^i}(\d y)\d t+\sqrt{\beta}\d
W_t^i+\int_{\R^d}\tilde{\sigma}(X_t^{i},y)\L_{X_t^i}(\d y)\d B_t^i.
\end{split}\end{equation}

Compared with the non-interacting noise case in \cite{DEGZ,GBMEJP,HYY,LMW,LWZ}, i.e. $\tilde{\sigma}=0$ or $\tilde{\sigma}(x,y)$ only depends on $x$, there exists essential difficulty in the study of uniform in time propagation of chaos in $\W_1$ for \eqref{Inm} due to the existence of interacting diffusion coefficient $\tilde{\sigma}$. The trick of asymptotic reflection coupling seems unavailable. Let us show the difficulty in the following.

In fact, for any $\varepsilon\in(0,1]$, let $\pi_R^\varepsilon\in[0,1]$ and $\pi_S^\varepsilon$ be two Lipschitz continuous functions on $[0,\infty)$ satisfying
\begin{align*}\pi_R^\varepsilon(x)=\left\{
      \begin{array}{ll}
        1, & \hbox{$x\geq \varepsilon$;} \\
        0, & \hbox{$x\leq \frac{\varepsilon}{2}$}
      \end{array}
    \right.,\ \ (\pi_R^\varepsilon)^2+(\pi_S^\varepsilon)^2=1.
\end{align*}
Let $\{\tilde{W}^i_t\}_{i\geq 1}$ be independent $d$-dimensional Brownian motions independent of $\{W_t^i,B_t^i\}_{i\geq 1}$. Set $\mu_t=\L_{X_t^i}$.
Construct
\begin{equation*}\begin{split}
\d \tilde{X}_t^{i}&=b^{(0)}(\tilde{X}_t^i)\d t+\int_{\R^d}b^{(1)}(\tilde{X}_t^i,y)\mu_t(\d y)\d t\\
&+\sqrt{\beta}\pi_R^\varepsilon(|\tilde{Z}_t^{i,N}|)\d
W_t^i+\sqrt{\beta}\pi_S^\varepsilon(|\tilde{Z}_t^{i,N}|)\d \tilde{W}^i_t+\int_{\R^d}\tilde{\sigma}(\tilde{X}_t^{i},y)\mu_t(\d y)\d B_t^i,
\end{split}\end{equation*}
and the asymptotic reflection coupling process
\begin{equation*}\begin{split}
\d \tilde{X}_t^{i,N}&=b^{(0)}(\tilde{X}_t^{i,N})\d t+\frac{1}{N}\sum_{j=1}^Nb^{(1)}(\tilde{X}_t^{i,N},\tilde{X}_t^{j,N})\d t\\
&+\sqrt{\beta}\pi_R^\varepsilon(|\tilde{Z}_t^{i,N}|)(I_{d\times d}-2\tilde{U}_t^{i,N}\otimes \tilde{U}_t^{i,N})\d
W_t^i\\
&+\sqrt{\beta}\pi_S^\varepsilon(|\tilde{Z}_t^{i,N}|)\d \tilde{W}^i_t
+\frac{1}{N}\sum_{j=1}^N\tilde{\sigma}(\tilde{X}_t^{i,N}, \tilde{X}_t^{j,N})\d B_t^i,
\end{split}\end{equation*}
where  $\tilde{Z}_t^{i,N}=\tilde{X}_t^i-\tilde{X}_t^{i,N}$, $\tilde{U}_t^{i,N}=\frac{\tilde{Z}_t^{i,N}}{|\tilde{Z}_t^{i,N}|}1_{\{|\tilde{Z}_t^{i,N}|\neq 0\}}$, $\tilde{X}_0^{i,N}=X_0^{i,N}, \tilde{X}_0^{i}=X_0^{i}, 1\leq i\leq N$.
 By It\^{o}-Tanaka's formula for $|\tilde{X}_t^{i}-\tilde{X}_t^{i,N}|$, a singular term
$$ \frac{1}{2}\frac{\|\int_{\R^d}\tilde{\sigma}(\tilde{X}_t^i,y)\mu_t(\d y)-\frac{1}{N}\sum_{j=1}^N\tilde{\sigma}(\tilde{X}_t^{i,N},\tilde{X}_t^{j,N})\|_{HS}^2}{|\tilde{X}^i_t -\tilde{X}_t^{i,N}|}1_{\{|\tilde{X}^i_t -\tilde{X}_t^{i,N}|\neq 0\}}$$
appears. This leads us rather difficult to derive estimate for $\E|\tilde{X}_t^i-\tilde{X}_t^{i,N}|$. To overcome this difficulty, we will introduce an auxiliary process $\bar X_t^i$, which solves
\begin{equation*}\begin{split}
&\d \bar{X}_t^{i}=b^{(0)}(\bar{X}_t^i)\d t+\int_{\R^d}b^{(1)}(\bar{X}_t^i,y)\mu_t(\d y)\d t+\sqrt{\beta}\d
W_t^i+\int_{\R^d}\tilde{\sigma}(\bar{X}_t^{i},y)\mu_t(\d y)\d B_t^i,\\
&\bar{X}_0^i=X_0^{i,N}, 1\leq i\leq N.
\end{split}\end{equation*}
In view of the triangle inequality
$$\W_1(\L_{(X_t^i)_{1\leq i\leq N}}, \L_{(X_t^{i,N})_{1\leq i\leq N}} )\leq \W_1(\L_{(X_t^i)_{1\leq i\leq N}}, \L_{(\bar X_t^i)_{1\leq i\leq N}})+\W_1(\L_{(\bar X_t^i)_{1\leq i\leq N}}, \L_{(X_t^{i,N})_{1\leq i\leq N}} ),$$
it is alternative to estimate $\W_1(\L_{(X_t^i)_{1\leq i\leq N}}, \L_{(\bar X_t^i)_{1\leq i\leq N}})$ and $\W_1(\L_{(\bar X_t^i)_{1\leq i\leq N}}, \L_{(X_t^{i,N})_{1\leq i\leq N}} )$ respectively. The former one is not difficult to be handled by reflection coupling method since $X_t^i$ and $\bar{X}_t^i$ solve the same time-inhomogeneous classical SDEs. To deal with the latter one, we will adopt the Duhamel formula for two semigroups associated to two time-inhomogeneous diffusion processes with different coefficients on $(\R^d)^N$. To illustrate the idea, for simplicity, let $b^i:[0,\infty)\times\R^d\to\R^d,\sigma^i:[0,\infty)\times \R^d\to\R^d\otimes\R^n, i=1,2$ be measurable and consider
$$\d Z^i_{s,t}=b^i_t(Z_{s,t}^i)\d t+\sigma^i_t(Z_{s,t}^i)\d W_t,\ \ t\geq s\geq 0.$$
Denote $\{P_{s,t}^i\}_{0\leq s\leq t}$ the associated semigroup to $Z_{s,t}^i$ and $\scr L^i_t$ be the generator
$$\scr L^i_t=\<b^i_t,\nabla\>+\frac{1}{2}\mathrm{tr}(\sigma_t^i(\sigma_t^i)^\ast\nabla^2),\ \ i=1,2.$$
Let for instance $f\in C_b^2(\R^d)$, the set of all continuous functions on $\R^d$ with bounded and continuous up to $2$ order derivatives.
%Assume that the Kolmogorov forward equation for $P_{s,t}^1$ and the Kolmogorov backward equation for $P_{s,t}^2$ hold respectively, i.e.
%\begin{align}\label{KFB}\frac{\d P_{s,t}^1 f}{\d t}=P_{s,t}^1 \scr L_t^1 f,\ \ \frac{\d P_{s,t}^2 f}{\d s}=-\scr L_s^2 P_{s,t}^2  f.
%\end{align}
Then Duhamel formula is formulated as
\begin{align*}P_{0,t}^1 f-P_{0,t}^2 f=\int_0^t [P_{0,s}^1\{(\L^1_s-\L^2_s)P_{s,t}^2f\}]\d s,\ \ t\geq 0,
\end{align*}
which can date back to \cite[(3a)]{MS}.
 %We will apply a comparison formula for two different diffusion processes.
%When \eqref{E0} is well-posed, we denote $X_t^\gamma$ the solution to it from initial distribution $\gamma\in\scr P_2(\R^d)$.
%and let $P_t^\ast\gamma=\L_{X_t^\gamma}$ as well as
%$$P_tf(\gamma):=\E f (X_t^\gamma)=\int_{\R^d}f(x)( P_t^\ast\gamma)(\d x),\ \ f\in\scr B_b(\R^d).$$

%So, we have
%\begin{align*}
%|X_t-Y_t|^2&=|X_0-Y_0|^2+2\int_0^t\<b_s(X_s,\L_{X_s|\F_s^B})-b_s(Y_s,\L_{Y_s|\F_s^B}),X_s-Y_s\>\d s\\
%&\leq |X_0-Y_0|^2+c\int_0^t\left(\W_2(\L_{X_s|\F_s^B}, \L_{Y_s|\F_s^B})^2+|X_s-Y_s|^2\right)\d s.
%\end{align*}

The remaining of the paper is organized as follows: In section 2,  we state the main result on uniform in time  propagation of chaos in $\W_1$-distance for mean field particle system with interacting diffusion coefficients and partially dissipative drifts. In Section 3, the proof of main result is provided.

%\subsection{Entropy and total variation distance}
%\begin{align}\label{IIT}\d \hat{X}_t^i= b_t(\hat{X}_t^i+\eta_t, \mu_t^i)\d t+  \sigma_t(\hat{X}^i_t+\eta_t) \d W^i_t,\ \ t\in [0,T],\ \ 1\leq i\leq N
%\end{align}
%and the stochastic interacting particle system
%\begin{align}\label{ITR}\d X^{i,N}_t=b_t(X_t^{i,N}, \hat\mu_t^N)\d t+\sigma_t(X^{i,N}_t) \d W^i_t+ \tt \si_t\d  B_t,\ \ 1\leq i\leq N, X_0^{i,N}=X_0^i,
%\end{align}
%where  $\hat\mu_t^N$ is the empirical distribution of $X_t^{1,N},\cdots,X_t^{N,N}$, i.e.
%\begin{equation*}
% \hat\mu_t^N =\ff{1}{N}\sum_{j=1}^N\dd_{X_t^{j,N}}.
% \end{equation*}
%The diffusion cannot depend on the distribution variable. We need to estimate the conditional entropy given $B_t$, the coefficients before $B$ only depend on $t$. In this case, given $B$, they are independent. Then
%$$\mathrm{Ent}(\L_{X^{i,N}}|B,\L_{X^{i}}|B)\leq \frac{2}{N}\mathrm{Ent}(\L_{X^{N}}|B,\L_{X}|B)$$

\section{Main results}

To derive the uniform in time propagation of chaos in $L^1$-Wasserstein distance, we make the following assumptions.
\begin{enumerate}
\item[{\bf(A)}] There exists a constant $K_\sigma>0$ such that for any $x_1,x_2,y_1,y_2\in\R^d$,
\begin{align}\label{bslipsg}
&\frac{1}{2}\|\tilde{\sigma}(x_1,y_1)-\tilde{\sigma}(x_2,y_2)\|^2_{HS}\leq K_\sigma(|x_1-x_2|^2+|y_1-y_2|^2),\ \ \tilde{\sigma}\tilde{\sigma}^\ast\leq K_\sigma.
\end{align}
Moreover, there exists $K_b\geq0$ such that
\begin{align}\label{lip1a}|b^{(1)}(x,y)-b^{(1)}(\tilde{x},\tilde{y})|\leq K_b(|x-\tilde{x}|+|y-\tilde{y}|),\ \ x,\tilde{x},y,\tilde{y}\in\R^d.
\end{align}
In addition, $b^{(0)}$ is continuous and there exist $R>0$, $K_1\geq 0, K_2>0$  such that
\begin{align}\label{pdi}
&\langle x_1-x_2, b^{(0)}(x_1)-b^{(0)}(x_2)\rangle\leq \gamma(|x_1-x_2|)|x_1-x_2|, \ \ x_1,x_2\in\R^d
\end{align}
with
$$\gamma(r)=\left\{
  \begin{array}{ll}
K_1r, & \hbox{$r\leq R$;} \\
    \left\{-\frac{K_1+K_2}{R}(r-R)+K_1\right\}r, & \hbox{$R\leq r\leq 2R$;} \\
    -K_2r, & \hbox{$r>2R$.}
  \end{array}
\right.
$$

\end{enumerate}
Under {\bf(A)} \eqref{Inm} is well-posed and \eqref{MVi} is well-posed in $\scr P_1(\R^d)$.

Before moving on, we introduce some notations. For any $k\geq 1$, let
$$ C_b^k(\R^d):= \big\{f: \R^d\to  \R \  \text{has bounded and continuous up to $k$ order derivatives}\big\}. $$
For any $F\in C_b^1((\R^d)^N)$,
%the set of all differentiable function on $(\R^d)^N$,
$x^i\in\R^d, 1\leq i\leq N$, let $\nabla_i F(x^1,x^2,\cdots,x^N)$ denote the gradient of $F$ with respect to the $i$-th component $x^i$. Simply denote $\nabla^2_i=\nabla_i\nabla_i$.
For any $\mu\in\scr P_1(\R^d)$, let $P_t^\ast\mu$ be the distribution of $X_t^1$ with initial distribution $\mu$, and for any exchangeable $\mu^N\in\scr P((\R^d)^N)$, $1\leq k\leq N$, $(P_t^{k})^\ast\mu^N$
be the distribution of $(X_t^{i,N})_{1\leq i\leq k}$ with initial distribution $\mu^N$. Moreover, for any $\mu\in\scr P(\R^d)$, let $\mu^{\otimes k}$ denote the $k$ independent product of $\mu$, i.e. $\mu^{\otimes k}=\prod_{i=1}^k\mu$.
Throughout this section, let $\mu_t=\L_{X_t^i}$ for $\mu_0\in\scr P_1(\R^d)$, which is independent of $i$ due to the weak uniqueness of \eqref{MVi}.
For any $s\geq 0$, consider the time-inhomogeneous decoupled SDE
\begin{align}\label{deSDE}\nonumber\d X_{s,t}^{i,\mu,z}&=b^{(0)}(X_{s,t}^{i,\mu,z})\d t+\int_{\R^d}b^{(1)}(X_{s,t}^{i,\mu,z},y)\mu_t(\d y)\d t\\
&+\sqrt{\beta}\d
W_t^i+\int_{\R^d}\tilde{\sigma}(X_{s,t}^{i,\mu,z},y)\mu_t(\d y)\d B_t^i, \ \ t\geq s,i\geq 1
\end{align}
with $X_{s,s}^{i,\mu,z}=z\in\R^d$.
%It is easy to see that
%$$\mu_t=\int_{\R^d}\L_{X_{0,t}^{\mu,z}}\mu_0(\d z).$$
Let
$$P_{s,t}^{i,\mu} f(z):=\E f(X_{s,t}^{i,\mu,z}), \ \ f\in \scr B_b(\R^{d}),z\in\R^d,i\geq 1,0\leq s\leq t.$$
Since \eqref{deSDE} is well-posed so that $P_{s,t}^{i,\mu}$ does not depend on $i$ and we denote \begin{align*}P_{s,t}^{\mu}=P_{s,t}^{i,\mu},\ \ i\geq 1.
 \end{align*}
For any $k\geq 1$, $x=(x^1,x^2,\cdots,x^k)\in (\R^{d})^k, F\in \scr B_b((\R^{d})^k)$ and $s\in[0,t]$, define
\begin{align*}(P_{s,t}^\mu)^{\otimes k} F(x):=\E F(X_{s,t}^{1,\mu,x^1},X_{s,t}^{2,\mu,x^2},\cdots,X_{s,t}^{k,\mu,x^k}),\ \ 0\leq s\leq t.
\end{align*}
For simplicity, we write $P_{t}^\mu =P_{0,t}^\mu $. Moreover, for any $1\leq k\leq N$, $\mu_0^k\in\scr P((\R^d)^k)$, we denote $$[((P_{t}^\mu)^{\otimes k})^\ast\mu_0^k](A)=\int_{(\R^d)^k}\left((P_{t}^\mu)^{\otimes k} 1_{A}\right)(x)\mu_0^k(\d x), \ \ A\in\scr B((\R^d)^k).$$
For simplicity, denote $(P_{t}^\mu)^\ast=((P_{t}^\mu)^{\otimes 1})^\ast$.

We also need the backward Kolmogorov equation and gradient estimate for $P_{s,t}^\mu$.
 \begin{enumerate}
\item[{\bf(A')}] The backward Kolmogorov equation for $P_{s,t}^\mu$ holds, i.e.
\begin{align}\label{BKE}\frac{\d P_{s,t}^\mu f}{\d s}=-\scr L_s^\mu P^\mu_{s,t}f,\ \ f\in C_b^2(\R^d), 0\leq s\leq t
\end{align}
for
\begin{align*}\scr L_s^\mu&=\<b^{(0)},\nabla\>+\left\<\int_{\R^d}b^{(1)}(\cdot,y)\mu_s(\d y),\nabla\right\>\\
&+\frac{1}{2}\mathrm{tr}\left[\left(\beta I_{d\times d}+\int_{\R^d}\tilde{\sigma}(\cdot,y)\mu_s(\d y)\left(\int_{\R^d}\tilde{\sigma}(\cdot,y)\mu_s(\d y)\right)^\ast\right)\nabla^2\right].
\end{align*}
Moreover, there exists a constant $\mathbf{c_G}>0$ such that the gradient estimate holds:
\begin{align}\label{gra}|\nabla^i P_{s,t}^\mu f|\leq \mathbf{c_{G}}((t-s)\wedge 1)^{-\frac{i}{2}+\frac{1}{2}},\ \ [f]_1\leq 1, i=1,2, t>s\geq 0,
\end{align}
here $[f]_{1}:=\sup_{x\neq y}\frac{|f(x)-f(y)|}{|x-y|}$.

%There exists a constant $K_0, K_\sigma>0$ such that
%\begin{align}\label{bslip}
%&|b^{(0)}(x_1)-b^{(0)}(x_2)|\leq K_0|x_1-x_2|, \ \ x_1,x_2\in\R^d,
%\end{align}

\end{enumerate}
Let
\begin{align}\label{pac}\delta:=\int_0^\infty s\e^{\frac{1}{2\beta}\int_0^s\gamma(v)\d v}\d s,\ \ c_E:=\frac{K_2\delta}{2\beta},\ \ \lambda_0:=\frac{2\beta}{\delta} -\frac{K_2\delta}{2\beta}(K_b+K_\sigma).
\end{align}
For any $\mathbf{a}> 0,t\geq 0$ satisfying $6\sqrt{2}\sqrt{d}\mathbf{a}\mathbf{c_G} (1\vee\sqrt{t})\e^{\lambda_0 t}t^{\frac{1}{2}}<1$, define
\begin{align}\label{Gkt}G(\mathbf{a},t):=&\e^{-\lambda_0 t}c_E\left(1-6\sqrt{2}\sqrt{d}\mathbf{a}\mathbf{c_G} (1\vee\sqrt{t})\e^{\lambda_0 t}t^{\frac{1}{2}}\right)^{-1},\ \
\end{align}
and
\begin{align}\label{ka0}\kappa_0&=\sup\left\{\mathbf{a}>0: \inf_{t\geq 0: 6\sqrt{2}\sqrt{d}\mathbf{a}\mathbf{c_G} (1\vee\sqrt{t})\e^{\lambda_0 t}t^{\frac{1}{2}}<1}G(\mathbf{a},t)<1\right\}.
\end{align}

The main result is the following theorem.
\begin{thm}\label{POCin} Assume {\bf (A)} and {\bf (A')} with $$K_b+K_\sigma<\min\left(\frac{4\beta^2}{K_2\delta^2}, \frac{K_2}{2},\kappa_0\right).$$
Let $\mu_0^N\in\scr P_{1}((\R^d)^N)$ be exchangeable and $\mu_0\in \scr P_{2}(\R^d)$. Then there exist some constants $c,\lambda>0$ such that
\begin{align}\label{CMYa}\nonumber&\W_1((P_t^{k})^\ast\mu^N_0,(P_t^\ast\mu_0)^{\otimes k})\\
&\leq c\e^{-\lambda t}\frac{k}{N}\W_1(\mu_0^N,\mu_0^{\otimes N})+ck\{1+\{\mu_0(|\cdot|^{2})\}^{\frac{1}{2}}\}N^{-\frac{1}{2}},\ \ t\geq 0, 1\leq k\leq N.
\end{align}
%(2) If in particular, $b_t(x,\mu)=\int_{\R^d}\tilde{b}(x-y)\mu(\d x)$ for some Lipschitz continuous function  $\tilde{b}$, then for any $X_0^i$ with $\L_{X_0^i}\in\scr P_2$, the assertion in (1) holds for
%$\frac{1}{N}$ replacing $R_{d,q}(N)$.
\end{thm}
In the present non-degenerate case, to ensure {\bf(A')}, the drifts can be only assumed to be Lipschitz continuous, see for instance \cite[Theorem 1.2]{MPZ} and \cite[(2.20)]{HRW23}. So, we get the following corollary.
\begin{cor} Under the assumption of Theorem \ref{POCin} with {\bf(A')} replaced by the condition that there exists a constant $K_0>0$ such that
\begin{align*}
|b^{(0)}(x)-b^{(0)}(\tilde{x})|\leq K_0|x-\tilde{x}|,\ \ x,\tilde{x}\in\R^d,
\end{align*}
the assertions in Theorem \ref{POCin} hold.
\end{cor}
\section{Proof of Theorem \ref{POCin}}
\subsection{Preparations}
The following lemma is from \cite[Lemma 2.1]{HX23e}.
\begin{lem}\label{CTY} Let $(V,\|\cdot\|_V)$ be a Banach space. $(Z_i)_{i\geq 1}$ are i.i.d. $V$-valued random variables with $\E\|Z_1\|_V^{2}<\infty$ and $h:V\times V\to \R$ is  measurable and of at most linear growth, i.e. there exists a constant $K_h>0$ such that
\begin{align*}|h(v,\tilde{v})|\leq K_h(1+\|v\|_V+\|\tilde{v}\|_V),\ \ v,\tilde{v}\in V.
\end{align*}
Then there exists a constant $\tilde{c}>0$ only depending on $K_h$ such that
\begin{align*}\E\left|\frac{1}{N}\sum_{m=1}^N h(Z_1,Z_m)-\int_{V} h(Z_1,y)\L_{Z_1}(\d y)\right|\leq \tilde{c}\{1+\{\E\|Z_1\|_{V}^{2}\}^{\frac{1}{2}}\}N^{-\frac{1}{2}}.
\end{align*}
\end{lem}

Next, we provide a uniform in time estimate for the second moment of \eqref{MVi}.
\begin{lem}\label{Umo} Assume {\bf(A)} with $K_b+K_\sigma<\frac{K_2}{2}$. Then  there exists a constant $c_0>0$ such that
$$\sup_{t\geq 0}\left(1+\{(P_t^\ast\mu_0)(|\cdot|^2)\}^{\frac{1}{2}}\right)\leq c_0(1+\{\mu_0(|\cdot|^2)\}^{\frac{1}{2}}).$$
\end{lem}
\begin{proof}
By {\bf(A)}, we can find a constant $C_0>0$ such that
\begin{align*}&2\<x,b^{(0)}(x)\>+2\left\<x,\int_{\R^d}b^{(1)}(x,y)\mu_t(\d y)\right\>+\beta d+\left\|\int_{\R^d}\tilde{\sigma}(x,y)\mu_t(\d y)\right\|_{HS}^2\\
&\leq (2K_1+2K_2)|x|^21_{\{|x|\leq 2R\}}-2K_2|x|^2+2K_\sigma|x|^2+2K_\sigma\mu_t(|\cdot|^2)\\
&+2\<x,b^{(0)}(0)\>+\beta d+2\sqrt{2K_\sigma}\|\tilde{\sigma}(0,0)\|_{HS}(|x|+\mu_t(|\cdot|))+\|\tilde{\sigma}(0,0)\|_{HS}^2\\
&+ 2|x|K_b(|x|+\mu_t(|\cdot|))+2|x||b^{(1)}(0,0)|\\
&\leq (2K_1+2K_2)4R^2+\beta d+\|\tilde{\sigma}(0,0)\|_{HS}^2-(2K_2-4K_\sigma-4K_b)|x|^2\\
&+2|x|(|b^{(0)}(0)|+2\sqrt{2K_\sigma}\|\tilde{\sigma}(0,0)\|_{HS}+|b^{(1)}(0,0)|)\\
&+ 2\sqrt{2K_\sigma}\|\tilde{\sigma}(0,0)\|_{HS}\{\mu_t(|\cdot|)-|x|\}\\
&-2K_\sigma|x|^2+2K_\sigma\mu_t(|\cdot|^2)-K_b|x|^2+ K_b\mu_t(|\cdot|)^2\\
&\leq C_0-(K_2-2K_\sigma-2K_b)|x|^2+ 2\sqrt{2K_\sigma}\|\tilde{\sigma}(0,0)\|_{HS}\{\mu_t(|\cdot|)-|x|\}\\
&-(2K_\sigma+K_b)\{|x|^2-\mu_t(|\cdot|^2)\},\ \ x\in\R^d.
\end{align*}
This together with the It\^{o} formula gives
\begin{align*}
\d |X_t^{1}|^2&\leq C_0\d t- (K_2-2K_\sigma-2K_b)|X_t^{1}|^2\d t\\
&+2\sqrt{2K_\sigma}\|\tilde{\sigma}(0,0)\|_{HS}\{\mu_t(|\cdot|)-|X_t^1|\}\d t\\
&-(2K_\sigma+K_b)\{|X_t^1|^2-\mu_t(|\cdot|^2)\}\d t+\d \bar{M}_t,~t\ge0
\end{align*}
for some martingale $\bar{M}_t$.
Combining this with $K_b+K_\sigma<\frac{K_2}{2}$, we conclude that there exists a constant $C>0$ such that
$$\E|X_t^1|^{2}\leq C(1+\E|X_0^1|^2),\ \ t\geq 0.$$
In view of $\L_{X_t^1}=P_t^\ast\mu_0$, the proof is completed .
\end{proof}
As stated in the Introduction, to complete the proof of Theorem \ref{POCin}, it is sufficient to estimate
 $\W_1( ((P_{t}^\mu)^{\otimes N})^\ast\mu_0^N,(P_t^\ast\mu_0)^{\otimes N})$ and $\W_1((P_t^N)^\ast\mu_0^N, ((P_{t}^\mu)^{\otimes N})^\ast\mu_0^N)$, which will be provided in the following two lemmas respectively. For simplicity, set $\mu_t=P_t^\ast\mu_0$.
\begin{lem}\label{st1} Assume {\bf(A)} with
\begin{align}\label{kbs}K_b+K_\sigma<\frac{4\beta^2}{K_2\delta^2}
\end{align}
for $\delta$ be in \eqref{pac}.
Then for any exchangeable $\mu_0^N\in\scr P_1((\R^d)^N)$ and $\mu_0\in\scr P_1(\R^d)$, it holds
$$\W_1( ((P_{t}^\mu)^{\otimes N})^\ast\mu_0^N,(P_t^\ast\mu_0)^{\otimes N})\leq c_E\e^{-\lambda_0 t}\W_1(\mu_0^N, \mu_0^{\otimes N}),\ \ t\geq 0$$
for $c_E$ and $\lambda_0$ be in \eqref{pac}.
\end{lem}
\begin{proof}
We adopt the technique of reflection coupling to complete the proof. Construct
\begin{equation*}\begin{split}
\d \tilde{X}_t^{i}&=b^{(0)}(\tilde{X}_t^i)\d t+\int_{\R^d}b^{(1)}(\tilde{X}_t^i,y)\mu_t(\d y)\d t+\sqrt{\beta}\d
W_t^i+\int_{\R^d}\tilde{\sigma}(\tilde{X}_t^{i},y)\mu_t(\d y)\d B_t^i,
\end{split}\end{equation*}
and
\begin{equation*}\begin{split}
\d  \hat{X}_t^{i}&=b^{(0)}(\hat{X}_t^{i})\d t+\int_{\R^d}b^{(1)}(\hat{X}_t^i,y)\mu_t(\d y)\d t\\
&+\sqrt{\beta}(I_{d\times d}-2\tilde{U}_t^{i}\otimes \tilde{U}_t^{i}1_{\{t\leq \tau\}})\d
W_t^i
+\int_{\R^d}\tilde{\sigma}(\hat{X}_t^{i},y)\mu_t(\d y)\d B_t^i,
\end{split}\end{equation*}
where  $\tilde{Z}_t^{i}=\tilde{X}_t^i-\hat{X}_t^{i}$, $\tau=\inf\{t\geq 0:|\tilde{Z}_t^{i}|=0\}$, $\tilde{U}_t^{i}=\frac{\tilde{Z}_t^{i}}{|\tilde{Z}_t^{i}|}1_{\{|\tilde{Z}_t^{i}|\neq 0\}}$, $\L_{(\tilde{X}_0^{i})_{1\leq i\leq N}}=\mu_0^{\otimes N}$ and $\L_{(\hat{X}_0^{i})_{1\leq i\leq N}}=\mu_0^N$.
By the It\^{o}-Tanaka formula, \eqref{bslipsg}-\eqref{pdi}, we have
\begin{align*}
\d |\tilde{Z}_t^{i}|&\leq  \gamma(|\tilde{Z}_t^{i}|)\d t+(K_b+K_\sigma)|\tilde{Z}_t^{i}|\d t\\
&+\left\<\int_{\R^d}\left[\tilde{\sigma}(\tilde{X}_t^{i},y)-\tilde{\sigma}(\hat{X}_t^{i},y)\right]\mu_t(\d y)\d B_t^i,\frac{\tilde{Z}_t^{i}}{|\tilde{Z}_t^{i}|}\right\>+ 2\sqrt{\beta}\left\<\frac{\tilde{Z}_t^{i}}{|\tilde{Z}_t^{i}|},\d W_t^i\right\>,\ \ t<\tau.
\end{align*}
Define
$$f(r)=\int_0^r\e^{-\frac{1}{2\beta}\int_0^u\gamma(v)\d v}\int_u^\infty s\e^{\frac{1}{2\beta}\int_0^s\gamma(v)\d v}\d s\d u,\ \ r\geq 0.$$
Then it is easy to see from \eqref{pac} that
$$f'(0)=\int_0^\infty s\e^{\frac{1}{2\beta^2}\int_0^s\gamma(v)\d v}\d s=\delta,$$
and
\begin{align}\label{sed}f''(r)=-\frac{1}{2\beta}\gamma(r)f'(r)-r.
\end{align}
By \cite[Page 1054]{W23}, we have
\begin{align}\label{fii}f''(r)\leq 0, \ \ r\geq 0,
\end{align}
and
\begin{align}\label{cop}\frac{2\beta}{K_2}r\leq f(r)\leq \delta r.
\end{align}
%Therefore,
%$$\frac{1}{f'(0)}\W_f\leq \W_1(\mu,\nu)\leq \frac{K_2-K_b}{2\beta}\W_f. $$
By It\^{o}'s formula and \eqref{fii}, we have
\begin{align}\label{itf}
\nonumber\d f(|\tilde{Z}_t^{i}|)&\leq f'(|\tilde{Z}_t^{i}|)\gamma(|\tilde{Z}_t^{i}|)\d t+f'(|\tilde{Z}_t^{i}|)(K_b+K_\sigma)|\tilde{Z}_t^{i}|\d t+2\beta f''(|\tilde{Z}_t^{i}|)\d t\\
&+f'(|\tilde{Z}_t^{i}|)\left\<\int_{\R^d} \left[\tilde{\sigma}(\tilde{X}_t^{i},y)-\tilde{\sigma}(\hat{X}_t^{i},y)\right]\mu_t(\d y)\d B_t^i,\frac{\tilde{Z}_t^{i}}{|\tilde{Z}_t^{i}|}\right\>\\
\nonumber&+f'(|\tilde{Z}_t^{i}|)2\sqrt{\beta}\left\<\frac{\tilde{Z}_t^{i}}{|\tilde{Z}_t^{i}|},\d W_t^i\right\>,\ \ t<\tau.
\end{align}
It follows from \eqref{sed} that
\begin{align*}
&f'(|\tilde{Z}_t^{i}|)\gamma(|\tilde{Z}_t^{i}|)+2\beta f''(|\tilde{Z}_t^{i}|)= -2\beta |\tilde{Z}_t^{i}|.
\end{align*}
This combined with $\|f'\|_\infty=f'(0)=\delta$, \eqref{cop} and \eqref{itf} gives
\begin{align*}
\d f(|\tilde{Z}_t^{i}|)
&\leq -\left\{\frac{2\beta}{\delta} -\frac{K_2\delta}{2\beta}(K_b+K_\sigma)\right\}f(|\tilde{Z}_t^{i}|)\d t\\
&+f'(|\tilde{Z}_t^{i}|)\left\<\int_{\R^d} \left[\tilde{\sigma}(\tilde{X}_t^{i},y)-\tilde{\sigma}(\hat{X}_t^{i},y)\right]\mu_t(\d y)\d B_t^i,\frac{\tilde{Z}_t^{i}}{|\tilde{Z}_t^{i}|}\right\>\\
\nonumber&+f'(|\tilde{Z}_t^{i}|)2\sqrt{\beta} \left\<\frac{\tilde{Z}_t^{i}}{|\tilde{Z}_t^{i}|},\d W_t^i\right\>, \ \ t<\tau.
\end{align*}
Recall that $\lambda_0=\frac{2\beta}{\delta} -\frac{K_2\delta}{2\beta}(K_b+K_\sigma)$ is given in \eqref{pac}. Then \eqref{kbs} implies $\lambda_0>0$. Hence, it follows that
\begin{align*}
&\E [\e^{\lambda_0 t}f(|\tilde{Z}_{t}^{i}|)|\F_0]=\E [\e^{\lambda_0 (t\wedge\tau)}f(|\tilde{Z}_{t\wedge \tau}^{i}|)1_{t<\tau}|\F_0]\leq \E [\e^{\lambda_0 (t\wedge\tau)}f(|\tilde{Z}_{t\wedge \tau}^{i}|)|\F_0]\leq f(|\tilde{Z}_0^{i}|).
\end{align*}
So, we have
\begin{align}\label{feg}
&\E [f(|\tilde{Z}_t^{i}|)|\F_0]\leq \e^{-\lambda_0 t}f(|\tilde{Z}_0^{i}|).
\end{align}
Recall $c_E=\frac{K_2\delta}{2\beta}$ is defined in \eqref{pac}. Then it holds $c_E\geq 1$ due to \eqref{cop}. \eqref{feg} together with \eqref{cop} implies that
$$\W_1(\L_{\tilde{X}_t^i|\F_0}, \L_{\hat{X}_t^i|\F_0})\leq c_E\e^{-\lambda_0 t}|X_0^{i,N}-X_0^{i}|. $$
Since both $(\tilde{X}_t^i)_{1\leq i\leq N}$ and $(\hat{X}_t^i)_{1\leq i\leq N}$ are independent under $\P^0$, we get
$$\W_1(\L_{(\tilde{X}_t^i)_{1\leq i\leq N}|\F_0}, \L_{(\hat{X}_t^i)_{1\leq i\leq N}|\F_0})\leq c_E\e^{-\lambda_0 t}\sum_{i=1}^N|X_0^{i,N}-X_0^{i}|.$$
Taking expectation first and then taking infimum in all $(X_0^i)_{1\leq i\leq N}$ and $(X_0^{i,N})_{1\leq i\leq N}$ satisfying $\L_{(X_0^{i,N})_{1\leq i\leq N}}=\mu_0^N$ and $\L_{(X_0^{i})_{1\leq i\leq N}}=\mu_0^{\otimes N}$, we derive
\begin{align*}\W_1( ((P_{t}^\mu)^{\otimes N})^\ast\mu_0^N,(P_t^\ast\mu_0)^{\otimes N})=\W_1(\L_{(\tilde{X}_t^i)_{1\leq i\leq N}}, \L_{(\hat{X}_t^i)_{1\leq i\leq N}})\leq c_E\e^{-\lambda_0 t}\W_1(\mu_0^N, \mu_0^{\otimes N}).
\end{align*}
Therefore, the proof is completed.
\end{proof}
\begin{lem}\label{st2} Assume {\bf (A)} and {\bf(A')} with $K_b+K_\sigma<\frac{K_2}{2}$. Let $\mu_0^N\in\scr P_{1}((\R^d)^N)$ be exchangeable and $\mu_0\in \scr P_{2}(\R^d)$. Then there exists a constant $c_L>0$ such that
\begin{align*}
&\W_1((P_t^N)^\ast\mu_0^N, ((P_{t}^\mu)^{\otimes N})^\ast\mu_0^N)\\
&\leq 3K_b\mathbf{c_G}\int_0^t\W_1((P_s^{N})^\ast\mu^N_0,(P_s^\ast\mu_0)^{\otimes N})\d s\\
&+3\sqrt{2}K_\sigma \mathbf{c_G}\sqrt{d}\int_0^t((t-s)\wedge1)^{-\frac{1}{2}}\W_1((P_s^{N})^\ast\mu^N_0,(P_s^\ast\mu_0)^{\otimes N})\d s\\
&+c_L\mathbf{c_G}\left(t+\int_0^t(s\wedge1)^{-\frac{1}{2}}\d s\right)\{1+\{\mu_0(|\cdot|^{2})\}^{\frac{1}{2}}\} NN^{-\frac{1}{2}},\ \ t\geq 0.
\end{align*}
\end{lem}
\begin{proof}
For any $x=(x^1,x^2,\cdots,x^N)\in(\R^{d})^N, s\geq 0, 1\leq i\leq N$ and $F\in C_b^2((\R^d)^N)$, define
\begin{align*}(\scr L_s^\mu)^{i}F(x)
&=\<b^{(0)}(x^i),\nabla_iF(x)\>+\left\<\int_{\R^d}b^{(1)}(x^i,y)\mu_s(\d y),\nabla_iF(x)\right\>\\
&+\frac{1}{2}\mathrm{tr}\left[\left(\beta I_{d\times d}+\int_{\R^d}\tilde{\sigma}(x^i,y)\mu_s(\d y)\left(\int_{\R^d}\tilde{\sigma}(x^i,y)\mu_s(\d y)\right)^\ast\right)\nabla_i^2F(x)\right],
\end{align*}
and
\begin{align*}(\scr L_s^\mu)^{\otimes N}F(x)
&=\sum_{i=1}^N(\scr L_s^\mu)^{i}F(x).
\end{align*}
By the same argument to derive \cite[(3.10)]{HYY} from \cite[(3.9)]{HYY}, we deduce from \eqref{BKE} that
\begin{align}\label{KOL}\frac{\d (P_{s,t}^\mu)^{\otimes N} F}{\d s}=-(\scr L_s^\mu)^{\otimes N} (P^\mu_{s,t})^{\otimes N}F, \ \ 0\leq s\leq t, F\in C_b^2((\R^d)^N).
\end{align}
For any $x=(x^1,x^2,\cdots,x^N)\in(\R^d)^N$, let $$B^i_s(x)=\frac{1}{N}\sum_{m=1}^Nb^{(1)}(x^i,x^m) -\int_{\R^d}b^{(1)}(x^i,y)\mu_s(\d y),\ \ s\geq 0, 1\leq i\leq N,$$
and
      \begin{align*}\Sigma^i_s(x)&=\left(\frac{1}{N}\sum_{m=1}^N\tilde{\sigma}(x^i,x^m) \right) \left(\frac{1}{N}\sum_{m=1}^N\tilde{\sigma}(x^i,x^m)\right)^\ast\\
        &-\left(\int_{\R^d}\tilde{\sigma}(x^i,y)\mu_s(\d y)\right)\left(\int_{\R^d}\tilde{\sigma}(x^i,y)\mu_s(\d y)\right)^\ast,\ \ s\geq 0, 1\leq i\leq N.
        \end{align*}
 Combining \eqref{KOL} with It\^{o}'s formula, for any $t> 0$, $s\in[0,t]$ and $F\in C_b^2((\R^d)^N)$, we have
\begin{align*}
&\d [(P_{s,t}^\mu)^{\otimes N} F](X_s^{1,N},X_s^{2,N},\cdots,X_s^{N,N})\\
&=\left[-(\scr L_s^\mu)^{\otimes N} (P^\mu_{s,t})^{\otimes N}F\right](X_s^{1,N},X_s^{2,N},\cdots,X_s^{N,N})\d s\\
&+\sum_{i=1}^N\<b^{(0)}(X_s^{i,N}),\nabla _i[(P_{s,t}^\mu)^{\otimes N} F](X_s^{1,N},X_s^{2,N},\cdots,X_s^{N,N})\>\d s\\
&+\sum_{i=1}^N\left\<\frac{1}{N}\sum_{m=1}^Nb^{(1)}(X_s^{i,N},X_s^{m,N}),\nabla_i [(P_{s,t}^\mu)^{\otimes N} F](X_s^{1,N},X_s^{2,N},\cdots,X_s^{N,N})\right\>\d s\\
&+\frac{1}{2}\sum_{i=1}^N\mathrm{tr}\Bigg[\Bigg(\beta I_{d\times d}+\frac{1}{N}\sum_{m=1}^N\tilde{\sigma}(X_s^{i,N},X_s^{m,N})\left(\frac{1}{N}\sum_{m=1}^N \tilde{\sigma}(X_s^{i,N},X_s^{m,N})\right)^\ast\Bigg)\\
&\qquad\qquad\quad\times \nabla^2_i[(P_{s,t}^\mu)^{\otimes N} F](X_s^{1,N},X_s^{2,N},\cdots,X_s^{N,N})\Bigg]\d s+\d M_s\\
&=\sum_{i=1}^N\left\<B_s^i,\nabla_{i}[(P^\mu_{s,t})^{\otimes N}F]\right\>(X_s^{1,N},X_s^{2,N},\cdots,X_s^{N,N})\d s\\
&+\frac{1}{2}\sum_{i=1}^N\mathrm{tr}\left[(\Sigma_s^i\nabla^2_i[(P_{s,t}^\mu)^{\otimes N} F])(X_s^{1,N},X_s^{2,N},\cdots,X_s^{N,N})\right]\d s+\d M_s
\end{align*}
for some martingale $M_s$.
Integrating with respect to $s$ from $0$ to $t$ and taking expectation, for any $\eta\in(0,1]$, $F\in C_b^2((\R^d)^N)$ with $[F]_{1,1}\leq 1$, we arrive at
\begin{align}\label{DUH}
\nonumber&\int_{(\R^d)^N}F(x)\{(P_t^{N})^\ast\mu^N_0\}(\d x)-\int_{(\R^d)^N}\{(P_{t}^\mu)^{\otimes N} F\}(x)\mu^N_0(\d x)\\
\nonumber&=\int_0^t\sum_{i=1}^N\int_{(\R^d)^N}\left\<B^i_s(x),[\nabla_{i}(P^\mu_{s,t})^{\otimes N}F](x)\right\>\{(P_s^{N})^\ast\mu^N_0\}(\d x)\d s\\
\nonumber&+\frac{1}{2}\int_0^t\sum_{i=1}^N\int_{(\R^d)^N}\mathrm{tr}(\Sigma^i_s[\nabla_{i}^2(P^\mu_{s,t})^{\otimes N}F])(x)\{(P_s^{N})^\ast\mu^N_0\}(\d x)\d s\\
&\leq \mathbf{c_G}\int_0^t((t-s)\wedge 1)^{\frac{-1+1}{2}}\sum_{i=1}^N\int_{(\R^d)^N}|B^i_s(x)|\{(P_s^{N})^\ast\mu^N_0\}(\d x)\d s\\
\nonumber&+\frac{1}{2}\mathbf{c_G}\sqrt{d}\int_0^t((t-s)\wedge 1)^{-1+\frac{1}{2}}\sum_{i=1}^N\int_{(\R^d)^N}\|\Sigma^i_s(x)\|_{HS}\{(P_s^{N})^\ast\mu^N_0\}(\d x)\d s\\
\nonumber&=:I_1+I_{2},
    \end{align}
here we used the fact
\begin{align*}|\nabla^j_{i}(P^\mu_{s,t})^{\otimes N}F|\leq \mathbf{c_G}((t-s)\wedge 1)^{\frac{-j+1}{2}},\ \ j=1,2, 1\leq i\leq N,[F]_{1,1}\leq 1,
\end{align*}
which is not difficult to be derived from \eqref{gra}. Next, we estimate $I_1$ and $I_2$ respectively.
Observe that {\bf(A)} implies
$$\sum_{i=1}^N|B^i(x)|-\sum_{i=1}^N|B^i(\tilde{x})|\leq 3K_b\|x-\tilde{x}\|_{1,1},\ \ x,\tilde{x} \in(\R^d)^N,$$
and
$$\sum_{i=1}^N\|\Sigma^i_s(x)\|_{HS}-\sum_{i=1}^N\|\Sigma_s^i(\tilde{x})\|_{HS}\leq 6\sqrt{2}K_\sigma\|x-\tilde{x}\|_{1,1}\ \ x,\tilde{x} \in(\R^d)^N.$$
This together with \eqref{wetad} yields
\begin{align}\label{bdi}\nonumber&I_1=\mathbf{c_G}\int_0^t((t-s)\wedge 1)^{\frac{-1+1}{2}}\sum_{i=1}^N\int_{(\R^d)^N}|B^i_s(x)|\{(P_s^{N})^\ast\mu^N_0\}(\d x)\d s\\
\nonumber&=\mathbf{c_G}\int_0^t((t-s)\wedge 1)^{\frac{-1+1}{2}}\sum_{i=1}^N\int_{(\R^d)^N}|B^i_s(x)|\{(P_s^{N})^\ast\mu^N_0-(P_s^\ast\mu_0)^{\otimes N}\}(\d x)\d s\\
&+\mathbf{c_G}\int_0^t((t-s)\wedge 1)^{\frac{-1+1}{2}}\sum_{i=1}^N\int_{(\R^d)^N}|B^i_s(x)|\{(P_s^\ast\mu_0)^{\otimes N}\}(\d x)\d s\\
\nonumber&\leq 3K_b\mathbf{c_G}\int_0^t((t-s)\wedge 1)^{\frac{-1+1}{2}}\W_1((P_s^{N})^\ast\mu^N_0,(P_s^\ast\mu_0)^{\otimes N})\d s\\
\nonumber&+\mathbf{c_G}\int_0^t((t-s)\wedge 1)^{\frac{-1+1}{2}}\sum_{i=1}^N\int_{(\R^d)^N}|B^i_s(x)|\{(P_s^\ast\mu_0)^{\otimes N}\}(\d x)\d s,
\end{align}
and similarly,
\begin{align}\label{sdi}\nonumber&I_2=\frac{1}{2}\mathbf{c_G}\sqrt{d}\int_0^t((t-s)\wedge 1)^{-1+\frac{1}{2}}\sum_{i=1}^N\int_{(\R^d)^N}\|\Sigma^i_s(x)\|_{HS}\{(P_s^{N})^\ast\mu^N_0\}(\d x)\d s\\
&\leq 3\sqrt{2}K_\sigma \mathbf{c_G}\sqrt{d}\int_0^t((t-s)\wedge 1)^{-1+\frac{1}{2}}\W_1((P_s^{N})^\ast\mu^N_0,(P_s^\ast\mu_0)^{\otimes N})\d s\\
\nonumber&+\frac{1}{2}\mathbf{c_G}\sqrt{d}\int_0^t((t-s)\wedge 1)^{-1+\frac{1}{2}}\sum_{i=1}^N\int_{(\R^d)^N}\|\Sigma^i_s(x)\|_{HS}\{(P_s^\ast\mu_0)^{\otimes N}\}(\d x)\d s.
\end{align}
By Lemma \ref{CTY} and Lemma \ref{Umo}, we can find a constant $c_L>0$ such that
\begin{align}\label{ubi}\nonumber&\mathbf{c_G}\int_0^t((t-s)\wedge 1)^{\frac{-1+1}{2}}\sum_{i=1}^N\int_{(\R^d)^N}|B^i_s(x)|\{(P_s^\ast\mu_0)^{\otimes N}\}(\d x)\d s\\
&+\frac{1}{2}\mathbf{c_G}\sqrt{d}\int_0^t((t-s)\wedge 1)^{-1+\frac{1}{2}}\sum_{i=1}^N\int_{(\R^d)^N}\|\Sigma^i_s(x)\|_{HS}\{(P_s^\ast\mu_0)^{\otimes N}\}(\d x)\d s\\
\nonumber&\leq c_L\mathbf{c_G}\left(\int_0^t(s\wedge1)^{\frac{-1+1}{2}}\d s+\int_0^t(s\wedge1)^{-1+\frac{1}{2}}\d s\right)\{1+\{\mu_0(|\cdot|^{2})\}^{\frac{1}{2}}\} NN^{-\frac{1}{2}},\ \ t\geq 0.
\end{align}
 Substituting \eqref{bdi}-\eqref{ubi} into \eqref{DUH}, we arrive at
\begin{align*}
&\W_1((P_t^N)^\ast\mu_0^N, ((P_{t}^\mu)^{\otimes N})^\ast\mu_0^N)\\
&\leq 3K_b\mathbf{c_G}\int_0^t((t-s)\wedge1)^{\frac{-1+1}{2}}\W_1((P_s^{N})^\ast\mu^N_0,(P_s^\ast\mu_0)^{\otimes N})\d s\\
&+3\sqrt{2}K_\sigma \mathbf{c_G}\sqrt{d}\int_0^t((t-s)\wedge1)^{-1+\frac{1}{2}}\W_1((P_s^{N})^\ast\mu^N_0,(P_s^\ast\mu_0)^{\otimes N})\d s\\
&+c_L\mathbf{c_G}\left(\int_0^t(s\wedge1)^{\frac{-1+1}{2}}\d s+\int_0^t(s\wedge1)^{-1+\frac{1}{2}}\d s\right)\{1+\{\mu_0(|\cdot|^{2})\}^{\frac{1}{2}}\} NN^{-\frac{1}{2}},\ \ t\geq 0.
\end{align*}
So, we complete the proof.
\end{proof}
\subsection{Proof of Theorem \ref{POCin}}
With Lemma \ref{st1} and Lemma \ref{st2} in hand, we are in the position to complete the proof of Theorem \ref{POCin}.
%\begin{rem} The following conditions can ensure {\bf(A')}.
%\begin{enumerate}
%\item[{\bf(A1)}] There exists a constant $K_\sigma>0$ such that
%\begin{align*}
%&\|\nabla\tilde{\sigma}(\cdot,y)(x)\|+\|\nabla^2\tilde{\sigma}(\cdot,y)(x)\|\leq K_\sigma,\ \ x,y\in\R^d.
%\end{align*}
%\item[{\bf(A2)}] There exists a constant $K_b>0$ such that
%    \begin{align*}
%    &\|\nabla b^{(0)}\|+\|\nabla^2 b^{(0)}\|+\|\nabla b^{(1)}(\cdot,y)(x)\|+\|\nabla^2b^{(1)}(\cdot,y)(x)\|\leq K_b,\ \ x,y\in\R^d.
%    \end{align*}
%\end{enumerate}
%\end{rem}

%(2) If in particular, $b_t(x,\mu)=\int_{\R^d}\tilde{b}(x-y)\mu(\d x)$ for some Lipschitz continuous function  $\tilde{b}$, then for any $X_0^i$ with $\L_{X_0^i}\in\scr P_2$, the assertion in (1) holds for
%$\frac{1}{N}$ replacing $R_{d,q}(N)$.
\begin{proof}[Proof of Theorem \ref{POCin}]
Combining Lemma \ref{st1}, Lemma \ref{st2} with the triangle inequality, we arrive at
\begin{align}\label{trian}
\nonumber&\W_1((P_t^{N})^\ast\mu^N_0,(P_t^\ast\mu_0)^{\otimes N})\\
\nonumber&\leq c_E\e^{-\lambda_0 t}\W_1(\mu^N_0,\mu_0^{\otimes N})+3K_b\mathbf{c_G}\int_0^t\W_1((P_s^{N})^\ast\mu^N_0,(P_s^\ast\mu_0)^{\otimes N})\d s\\
&+3\sqrt{2}K_\sigma \mathbf{c_G}\sqrt{d}\int_0^t((t-s)\wedge1)^{-\frac{1}{2}}\W_1((P_s^{N})^\ast\mu^N_0,(P_s^\ast\mu_0)^{\otimes N})\d s\\
\nonumber&+c_L\mathbf{c_G}\left(t+\int_0^t(s\wedge1)^{-\frac{1}{2}}\d s\right)\{1+\{\mu_0(|\cdot|^{2})\}^{\frac{1}{2}}\} NN^{-\frac{1}{2}},\ \ t\geq 0.
\end{align}
In view of
\begin{align*}&3K_b\mathbf{c_G}\int_0^t\W_1((P_s^{N})^\ast\mu^N_0,(P_s^\ast\mu_0)^{\otimes N})\d s\\
&\leq 3K_b\mathbf{c_G}\sqrt{t}\int_0^t(t-s)^{-\frac{1}{2}}\W_1((P_s^{N})^\ast\mu^N_0,(P_s^\ast\mu_0)^{\otimes N})\d s,
\end{align*}
and
\begin{align*}
&3\sqrt{2}K_\sigma \mathbf{c_G}\sqrt{d}\int_0^t((t-s)\wedge1)^{-\frac{1}{2}}\W_1((P_s^{N})^\ast\mu^N_0,(P_s^\ast\mu_0)^{\otimes N})\d s\\
&=3\sqrt{2}K_\sigma \mathbf{c_G}\sqrt{d}\int_0^t\frac{((t-s)\wedge1)^{-\frac{1}{2}}}{(t-s)^{-\frac{1}{2}}}(t-s)^{-\frac{1}{2}}\W_1((P_s^{N})^\ast\mu^N_0,(P_s^\ast\mu_0)^{\otimes N})\d s\\
&\leq 3\sqrt{2}K_\sigma \mathbf{c_G}\sqrt{d}(1\vee\sqrt{t})\int_0^t(t-s)^{-\frac{1}{2}}\W_1((P_s^{N})^\ast\mu^N_0,(P_s^\ast\mu_0)^{\otimes N})\d s,
\end{align*}
we obtain from \eqref{trian} that
\begin{align*}
&\W_1((P_t^{N})^\ast\mu^N_0,(P_t^\ast\mu_0)^{\otimes N})\\
&\leq c_E\e^{-\lambda_0 t}\W_1(\mu^N_0,\mu_0^{\otimes N})+c_L\mathbf{c_G}\left(t+\int_0^t(s\wedge1)^{-\frac{1}{2}}\d s\right)\{1+\{\mu_0(|\cdot|^{2})\}^{\frac{1}{2}}\} NN^{-\frac{1}{2}}\\
&+3\sqrt{2}\sqrt{d}(K_b+K_\sigma)\mathbf{c_G} (1\vee\sqrt{t})\int_0^t(t-s)^{-\frac{1}{2}}\W_1((P_s^{N})^\ast\mu^N_0,(P_s^\ast\mu_0)^{\otimes N})\d s.
\end{align*}
This implies that for any $t>0$,
\begin{align*}
\nonumber&\sup_{s\in[0,t]}\e^{\lambda_0 s}\W_1((P_s^{N})^\ast\mu^N_0,(P_s^\ast\mu_0)^{\otimes N})\\
\nonumber&\leq c_E\W_1(\mu^N_0,\mu_0^{\otimes N})+c_L\mathbf{c_G}\e^{\lambda_0 t}\left(t+\int_0^t(s\wedge1)^{-\frac{1}{2}}\d s\right)\{1+\{\mu_0(|\cdot|^{2})\}^{\frac{1}{2}}\} NN^{-\frac{1}{2}}\\
&+6\sqrt{2}\sqrt{d}(K_b+K_\sigma)\mathbf{c_G} (1\vee\sqrt{t})\e^{\lambda_0 t}t^{\frac{1}{2}}\sup_{s\in[0,t]}\e^{\lambda_0 s}\W_1((P_s^{N})^\ast\mu^N_0,(P_s^\ast\mu_0)^{\otimes N}).
\end{align*}
So, when $6\sqrt{2}\sqrt{d}(K_b+K_\sigma)\mathbf{c_G} (1\vee\sqrt{t})\e^{\lambda_0 t}t^{\frac{1}{2}}<1$, we conclude that
\begin{align}\label{cyk}
\nonumber&\W_1((P_t^{N})^\ast\mu^N_0,(P_t^\ast\mu_0)^{\otimes N})\\
&\leq \e^{-\lambda_0 t}c_E\left(1-6\sqrt{2}\sqrt{d}(K_b+K_\sigma)\mathbf{c_G} (1\vee\sqrt{t})\e^{\lambda_0 t}t^{\frac{1}{2}}\right)^{-1}\W_1(\mu^N_0,\mu_0^{\otimes N})\\
\nonumber&+\e^{-\lambda_0 t}\left(1-6\sqrt{2}\sqrt{d}(K_b+K_\sigma)\mathbf{c_G} (1\vee\sqrt{t})\e^{\lambda_0 t}t^{\frac{1}{2}}\right)^{-1}c_L\mathbf{c_G}\e^{\lambda_0 t}\left(t+\int_0^t(s\wedge1)^{-\frac{1}{2}}\d s\right)\\
\nonumber&\qquad\quad\times\{1+\{\mu_0(|\cdot|^{2})\}^{\frac{1}{2}}\} NN^{-\frac{1}{2}}.
\end{align}
Recall that $G(\mathbf{a},t)$ and $\lambda_0$ are defined in \eqref{Gkt} and \eqref{pac} respectively. Note that for any $\mathbf{a}>0$, $G(\mathbf{a},\cdot)$ is continuous on the set $\{t\geq 0: 6\sqrt{2}\sqrt{d}\mathbf{a}\mathbf{c_G} (1\vee\sqrt{t})\e^{\lambda_0 t}t^{\frac{1}{2}}<1\}$.
Consequently, when $K_b+K_\sigma\in(0,\kappa_0)$ with $\kappa_0$ defined in \eqref{ka0}, we can find $\hat{t}>0$ such that
\begin{align*}\alpha:=\e^{-\lambda_0 \hat{t}}c_E\left(1-6\sqrt{2}\sqrt{d}(K_b+K_\sigma)\mathbf{c_G} (1\vee\sqrt{\hat{t}})\e^{\lambda_0 \hat{t}}\hat{t}^{\frac{1}{2}}\right)^{-1}<1.
\end{align*}
Hence, we derive from \eqref{cyk} that there exist constants $\tilde{c}_0>0$ and $c_3>0$ such that
\begin{align}\label{CMYa1}&\W_1((P_{\hat{t}}^{N})^\ast\mu^N_0,(P_{\hat{t}}^\ast\mu_0)^{\otimes N})\leq \alpha\W_1(\mu_0^N,\mu_0^{\otimes N})+\tilde{c}_0\{1+\{\mu_0(|\cdot|^{2})\}^{\frac{1}{2}}\}NN^{-\frac{1}{2}},
\end{align}
and
\begin{align}\label{w1f}\sup_{t\in[0,\hat{t}]}\W_1((P_t^{N})^\ast\mu^N_0,(P_t^\ast\mu_0)^{\otimes N})\leq c_3\W_1(\mu^N_0,\mu_0^{\otimes N})+c_3\{1+\{\mu_0(|\cdot|^{2})\}^{\frac{1}{2}}\}NN^{-\frac{1}{2}}.
\end{align}
By \eqref{CMYa1}-\eqref{w1f}, the semigroup property $(P_{t+s}^{N})^\ast=(P_{t}^{N})^\ast(P_{s}^{N})^\ast$ as well as $P_{t+s}^\ast=P_{t}^\ast P_{s}^\ast$ and $\sup_{t\geq 0}\mu_t(|\cdot|^2)<\infty$ due to Lemma \ref{Umo}, we can find constants $c,\lambda>0$ such that
\begin{align}\label{CMYan}\W_1((P_t^{N})^\ast\mu^N_0,(P_t^\ast\mu_0)^{\otimes N})\leq c\e^{-\lambda t}\W_1(\mu_0^N,\mu_0^{\otimes N})+c\{1+\{\mu_0(|\cdot|^{2})\}^{\frac{1}{2}}\}NN^{-\frac{1}{2}},\ \ t\geq 0.
\end{align}
In fact, for any $t\geq \hat{t}$, let $\lfloor\frac{t}{\hat{t}}\rfloor$ be the integer part of $\frac{t}{\hat{t}}$. It follows from \eqref{CMYa1}-\eqref{w1f}, Lemma \ref{Umo} and the semigroup property $(P_{t+s}^{N})^\ast=(P_{t}^{N})^\ast(P_{s}^{N})^\ast$, $P_{t+s}^\ast=P_{t}^\ast P_{s}^\ast$ that
\begin{align}\label{CMYa2}\nonumber&\W_1((P_{t}^{N})^\ast\mu^N_0,(P_{t}^\ast\mu_0)^{\otimes N})\\
&\leq \alpha^{\lfloor\frac{t}{\hat{t}}\rfloor} \W_1((P_{t-\hat{t}\lfloor\frac{t}{\hat{t}}\rfloor}^{N})^\ast\mu^N_0,(P_{t-\hat{t}\lfloor\frac{t}{\hat{t}}\rfloor}^\ast\mu_0)^{\otimes N})+\sum_{\ell=1}^{\lfloor\frac{t}{\hat{t}}\rfloor}\alpha^{\ell-1}\bar{C}\{1+\{\mu_0(|\cdot|^{2})\}^{\frac{1}{2}}\} NN^{-\frac{1}{2}}\\
\nonumber&\leq c_3\alpha^{\lfloor\frac{t}{\hat{t}}\rfloor}\W_1(\mu^N_0,\mu_0^{\otimes N})+\left(c_3+\frac{\bar{C}}{1-\alpha}\right)\{1+\{\mu_0(|\cdot|^{2})\}^{\frac{1}{2}}\}NN^{-\frac{1}{2}},\ \ t\geq\hat{t}
\end{align}
for some constant $\bar{C}>0$. Combining \eqref{CMYa2} with \eqref{w1f}, we deduce \eqref{CMYan} for $\lambda=-\hat{t}^{-1}\log\alpha$.
Finally, we derive \eqref{CMYa} by combining \eqref{CMYan} and the fact
$$\W_1((P_t^{k})^\ast\mu^N_0,(P_t^\ast\mu_0)^{\otimes k})\leq \frac{k}{N}\W_1((P_t^{N})^\ast\mu^N_0,(P_t^\ast\mu_0)^{\otimes N}), \ \ 1\leq k\leq N.$$
Hence, the proof is completed.
\end{proof}

\end{document}